\documentclass[12pt]{article}
\usepackage{amsmath}
\usepackage{amsfonts}
\usepackage{amsthm}
\usepackage{hyperref}

\usepackage{indentfirst}
\usepackage{mathtools}

\usepackage{graphicx}

\usepackage{cleveref}
\usepackage{thmtools}
\usepackage{thm-restate}
\crefname{lemma}{Lemma}{Lemmas}
\crefname{theorem}{Theorem}{Theorems}
\crefname{prop}{Proposition}{Propositions}
\crefname{cor}{Corollary}{Corollaries}

\newtheorem{theorem}{Theorem}
\newtheorem{lemma}[theorem]{Lemma}

\newtheorem{defn}[theorem]{Definition}

\numberwithin{equation}{section}
\numberwithin{lemma}{section}
\numberwithin{theorem}{section}
\numberwithin{cor}{section}
\numberwithin{defn}{section}

\usepackage{blindtext}

\begin{document}

\date{}

\author{Martin Ortiz Ramirez}
\title{Lattice points in d-dimensional spherical segments}

\maketitle

\begin{abstract}
We study lattice points in $d$-dimensional spheres, and count their number in thin spherical segments. We found an upper bound depending only on the radius of the sphere and opening angle of the segment. To obtain this bound we slice the segment by hyperplanes of rational direction, and
 then cover an arbitrary segment with one having rational direction. Diophantine approximation can be used to obtain the best rational direction possible.   \\
 {\bf Keywords:} discrete geometry, lattice points on spheres, Diophantine approximation. \\
 {\bf MSC(2010):}  11P21, 11K60.

\end{abstract}
\medskip

\section{Introduction}

\indent The topic of lattice points in $d$-spheres has been extensively studied for centuries, starting with the two dimensional case of Gauss's circle problem, which asks for the number of lattice points $N(r)$ inside a circle of radius $r$, with  $r^2=n$ an integer. Gauss already knew that by geometric 
  considerations $N(r)$ is close to the area of the circle $\pi r^2$, and the problem transitioned to the study of the asymptotic behaviour of $E(r)=\lvert N(r)-\pi r^2\rvert$. One natural question asks for the the infimum $\mu$ of the $\alpha$ such that $E(r)=\mathcal{O}(r^\alpha)$. Gauss proved that $\mu \le 1$ by considering the length of the circumference, Sierpinski improved it to $\mu \le 2/3$, and Van der Corput showed that $\mu < 2/3$ 
   \cite[p.20-22]{grosswald}. They also conjectured that $\mu=1/2$. The best upper bound to date is $\mu \le 517/824$ by Bourgain and Watt \cite{bourgain_gauss}. 
     As for a lower bound, Hardy and Landau independently proved that $\mu \ge 1/2$. 

\par

Gauss's circle problem can be interpreted by considering the sum of all the $r_2(m)$ for $m \le n$, where $r_2(m)$ is the number of integral solutions to $x^2+y^2=m$. Unlike Gauss's problem, where obtaining the correct order of magnitude is easy, it is more difficult for $r_2(m)$, since it vanishes for arbitrarily large $m$. 
Relating $r_2(m)$ to the number of divisors of $m$ \cite[Ch 16.10]{hardy-wright} yields $r_2(m) \ll m^\epsilon$ for any $\epsilon >0$  \footnote{Here and in what follows $f(m)\ll g(m)$ means that $f(m)/g(m)$ is bounded as $m$ tends to infinity.}. Consider the number of lattice points on a short circular arc. Cilleruelo and Cordoba \cite{cilleruelocordoba} proved that on a circle of radius $R$, an arc of length no greater than $\sqrt{2}R^{1/2-1/(4\lfloor l/2\rfloor+2)}$  contains at most $l$ lattice points. 
Cilleruelo and Granville \cite{cilleruelogranville} proved that on an arc of length 
 less than $(40+\frac{40}{3}\sqrt{10})^{1/3}R^{1/3}$ there are at most $3$ lattice points, an improvement on a result by Jarnik \cite{jarnik}, that in an arc of length less than or equal to than $R^{1/3}$ there are at most two lattice points.

\par

Moving to $3$ dimensions, $S(R)$ denoting the number of lattice points inside a sphere of radius $R$, we have the analogue of Gauss's circle problem. By the same geometric considerations $S(R)$ is of the order of $\frac{4}{3}\pi R^3$, the volume of the sphere. Hence we are interested in $E(R)=\lvert \frac{4}{3}\pi R^3 - S(R) \rvert$. Szego \cite{szego} proved that $E(R)=o(R(logR)^{1/2})$ is not true \footnote{Here and elsewhere f(R)=o(g(R)) means that f(R)/g(R) tends to $0$ as $R$ tends to infinity.}. Regarding upper bounds of the form $E(R)=\mathcal{O}(R^\theta)$, Heath-Brown \cite{heathbrownsphere} proved that it holds true with $\theta= 21/16$. 

\par 

As in the two dimensional case, the number $r_3(R^2)$ of lattice points on the sphere of radius $R$, $R^2=n$ still vanishes for arbitrarily large $n$, of the form $4^a(8k+7)$ (Legendre's theorem \cite[Ch. 4]{grosswald}), although Walfisz \cite{walfisz} proved that $r_3(R^2) \gg R\log{\log{R}}$  holds for infinitely many $n$. We also have the upper bound $r_3(R^2) \ll R^{1+\epsilon}$ for any $\epsilon >0$ and the lower bound $r_3(R^2) \gg R^{1-\epsilon}$ for $n \neq 0,4,7 \; \text{mod} \; 8$ \cite[§1]{BSR} i.e. when there are primitive solutions to $x^2+y^2+z^2=n$ (solutions where the greatest common factor of $x$, $y$ and $z$ is $1$). The distribution of lattice points on the sphere has also been a topic of study. In the 1950s Linnik proved that as $n$ tends to infinity amongst $n$ square-free and $n\equiv \pm 1 \; \text{mod}\;5$ \footnote{In fact, Linnik aslo proved the statement replacing the constraint mod $5$ by a similar one mod $p$, for any fixed prime $p$.}, the projections of the lattice points onto the unit sphere becomes equidistributed. Duke \cite{Duke1988} proved the result removing the constraint $n \equiv \pm 1 \; \text{mod} \; 5$ after thirty years. 

\par
A three dimensional analogue of an arc is a \textit{cap}, the surface obtained by cutting the sphere with a plane, or equivalently, the intersection of $RS^2$ and sphere of radius $\lambda$ centered at a point of the sphere. Bourgain and Rudnick \cite{rudnick} proved using a theorem by Jarnik \cite{jarnik} that the maximal number $F_3(R,\lambda)$ of lattice points in a cap of radius $\lambda$ satisfies 
\[F_3(R,\lambda) \ll R^{\epsilon}\left(1+\frac{\lambda^2}{R^{1/2}}\right) \;\; \text{for all} \; \epsilon >0.\]
 This is only nontrivial 
  for small caps having $\lambda \ll R^{1/2}$, because other methods already give $F_3(R,\lambda) \ll R^{\epsilon}(1+\lambda)$ \cite[Lemma 2.2]{rudnick}. 
   For bigger caps Bourgain and Rudnick also proved \cite[Proposition 1.3]{rudnick} 
 \[F_3(R,\lambda) \ll R^{\epsilon}\left(1+\lambda \left(\frac{\lambda}{R}\right)^{\eta}\right) \;\; \text{for any} \;\; 0 < \eta < \frac{1}{15}.\]  
\indent If we slice the sphere by two parallel planes we obtain a \textit{spherical segment}. Maffucci \cite[Proposition 6.2]{maffucci} gave a bound (see the paragraph before \Cref{mainprop})
  for the number of lattice points in segments given their opening angle, a theorem that we generalise in this paper.

\par
The four dimensional case still shows an erratic behaviour, for instance $r_4(2^n)=24$ for all $n$. For $d$ greater than $4$ the behaviour of $r_d$ is much nicer, as the order of magnitude of $r_d$ is known to be $r_d(R^2)\approx R^{d-2}$, via the circle method \cite[Chapters 10-12]{grosswald}. By similar methods Bourgain-Rudnick \cite[Appendix A]{rudnick} also proved bounds for the number of lattice points in $d$-dimensional spherical caps: 
\[F_d(R,\lambda) \ll R^{\epsilon}\left(\frac{\lambda^{d-1}}{R}+\lambda^{d-3}\right) \; \text{for} \; d \ge 5.\]
The proof for $d\ge 5$ is considerably different from its three dimensional analogue. The former uses analytic tools such as the circle method. The latter is more geometric in nature, and uses Diophantine approximation, in a similar way to this paper. 

\par

While interesting in its own right, lattice points problems  have found applications in other fields. Recently they have been applied to the study of arithmetic waves, with results by Oravecz-Rudnick-Wigman \cite{leray}, Krishnapur-Kurlberg-Wigman \cite{wigman}, Rossi-Wigman \cite{rossi-wigman} and Benatar-Maffucci \cite{benatar-maffucci} among others.

Before stating the main results, we first need to define some key concepts. Let $S^{d-1}$ denote the unit $(d-1)$-sphere in $\mathbb{R}^d$, and $\overline{B}(\alpha,r)$ the closed ball of radius $r$ around $\alpha \in \mathbb{R}^d$. For a vector $\beta \in \mathbb{R}^d$, $\lvert \beta \rvert$ will denote its euclidean norm.
\begin{defn} \label{def1}
A spherical cap $T\subseteq RS^{d-1}$ of direction $\beta \in \mathbb{R}^d$, $\lvert \beta \rvert=1$, and radius $r$ is defined as 
\[
T=RS^{d-1} \cap \overline{B}(R\beta,r).
\]
We then define a spherical segment $ S \subseteq RS^{d-1}$ as $S=T_1\setminus T_2$, where each $T_{i}$ is a spherical cap of direction $\beta$ and radius $r_{i}$, $r_1>r_2$. We also say that the direction of any of the $T_i$ is the direction of $S$.
\end{defn}

 We will define $RS^{d-1} \cap \{x\in \mathbb{R}^d : \lvert x-R\beta \rvert=r\}$ as the $\textit{base}$ of $T$, we will see in the next section that it is a $(d-2)$-sphere in a hyperplane orthogonal to $\beta$. In a $(d-1)$-sphere with centre $O$ and radius $r$, two points $P_1,P_2$ are said to be $\textit{antipodal}$ if $P_1=O+v$, $P_2=O-v$, where $v$ has norm $r$. Hence the following definition. 
\begin{defn} \label{def2}
The \textit{opening angle} of a spherical cap is $\angle{POQ}$, where $O$ is the origin and $P$,$Q$ are antipodal points 
 of the $(d-2)$-sphere that is the base. The \textit{opening angle $\theta$ of a segment} $S$ is $\theta_1-\theta_2$, where $\theta_i$ is the opening angle of $T_i$.
\end{defn}
 We define $\psi(R,\theta)$ to be the maximal number of lattice points in segments $T \subseteq RS^{d-1}$ of opening angle $\theta$.
Our method will involve slicing the segment by hyperplanes. In this context, $\kappa_d(R)$ \footnote{To the best of my knowledge, an upper bound for $\kappa_d$ when $d\ge 4$ has not been explicitly studied in the literature, but a bound like $\kappa_d(R) \ll R^{d-3+\epsilon}$ would be expected to be amenable to the circle method for $d\ge 5$, since the intersection of a hyperplane and a $(d-1)$-sphere is a $(d-2)$-sphere.} is the maximal number of integer points in the intersection of $RS^{d-1}$ and a \textit{rational hyperplane}, that is, a hyperplane that can be defined by a the equation of the form $a_1x_1+a_2x_2+\ldots a_d x_d=a$, for $a_i \in \mathbb{Z}$. We are now in position to state the main results of the paper. The first one is a generalisation of the following result by Maffucci \cite[Proposition 6.2] {maffucci} for the three dimensional case:
\[ \psi \ll \kappa_3(R)(1+R\theta^{1/3}) \;\; \text{as} \;\; \theta \to 0 \text{ along with } R \to \infty.\]
In all future statements $\Psi(R,\theta) \ll f(R,\theta)$ as $\theta \to 0$ will mean that $\theta$ tends to $0$ together with $R \to \infty$. Furthermore, unless otherwise stated, the result is understood to be uniform, that is, the implied constant does not depend on $R$ or $\theta$.
\begin{theorem} \label{mainprop}
Let $\psi(R,\theta)$ be the maximal number of lattice points on spherical segments in $RS^{d-1}$ with opening angle $\theta$. Then 
\[\psi \ll \kappa_d(R)(1+R\theta^{1/d})\]
as $\theta \to 0$.
\end{theorem}

We can improve \Cref{mainprop} if we are given more information about the direction of the segment.
\begin{defn} \label{def3}
For $\beta \in \mathbb{R}^d$, $\beta$ is said to have $s$ \textit{rational quotients} if there exists a non-zero $k \in \mathbb{R}$ such that
$k\beta$
has exactly $s+1$ rational coordinates, and this number of rational coordinates is maximal amongst all the non-zero $k$. 
\end{defn}
We then fix some direction $\beta$ and ask for the number of lattice points in segments of radius $R$ and opening angle $\theta$. The case $s=0$ is already covered in \Cref{mainprop}.

\begin{theorem} \label{thm1.2}
Let $S\subseteq RS^{d-1}$ be a spherical segment of opening angle $\theta$ and direction $\beta$ having $s$ rational quotients, $1\le s\le d-1$. Let $\psi(R,\theta,\beta)$ be the number of lattice points in $S$. Then 
\[
\psi \ll_\beta \kappa_d(R)\left(1+R\theta^{\frac{1}{d-s}}\right)
\]
as $\theta \to 0$. \footnote{Here $\ll_\beta$ means that the implied constant depends on $\beta$.}

\end{theorem}

Next, we prove some general geometric facts about segments that will be useful later on. On Section $3$ we will state all the relevant lemmas and prove \Cref{mainprop}. Section $4$ is devoted to the proof of such lemmas and we prove \Cref{thm1.2} on Section $5$.

\section{Some geometric considerations}
In this section we introduce a few geometric features of spherical segments that will be used later on. Recall \Cref{def1}, we now prove that the base of each $T_i$ is indeed a $(d-2)$-sphere. Let $v \cdot w$ denote the dot product of two vectors in $\mathbb{R}^d$. The intersection of $\{ \lvert x - R\beta \rvert =r_i \}$ and $\{\lvert x \rvert=R\}$ lies in the plane $\beta \cdot x=R-\frac{r_i^2}{2R}\coloneqq \lambda_i$ since 
\[\lvert x - R\beta \rvert^2=R^2+ \lvert x \rvert^2-2R\beta \cdot x=r_i^2\]
 holds on the intersection.
Then 
\[\lvert \lambda_i\beta-x\rvert^2=\lambda_i^2+R^2-2\lambda_i\beta \cdot x=R^2-\lambda_i^2=r_i^2-\frac{r_i^4}{4R^2},\]
 so that $S$ has two bases $B_1$ and $B_2$ that are $(d-2)$-spheres lying on the hyperplanes $\beta \cdot x = \lambda_i$
 , of radii 
 \[k_i^2\coloneqq r_i^2-\frac{r_i^4}{4R^2}.\] 
\indent In \Cref{def2}, $POQ$ is an isosceles triangle of side lengths $R$ and $2k_i$, so $\theta$ is well-defined.
Because $P$ and $Q$ are antipodal points, $O, P, Q$ and $R\beta$ are coplanar, hence by basic geometry we have that 
\begin{equation} \label{rtheta}
r_i=2R\sin(\theta_i/4).
\end{equation}
Another parameter that will be useful is the \textit{height} of the segment, the distance between the two bases $B_1$ and $B_2$ 
  \begin{equation} \label{height}
  h=\left \lvert \frac{\beta\lambda_1}{\lvert \beta \rvert^2}-\frac{\beta\lambda_2}{\lvert \beta \rvert^2}\right \rvert=\frac{1}{\lvert \beta \rvert}\lvert \lambda_1-\lambda_2\rvert=\frac{1}{2R}(r_1^2-r_2^2).
  \end{equation}

\section{Proof of Theorem 1.3}

We will first state the necessary lemmas, that will be proved in the next section. 
\begin{lemma}
\label{lem:lemma1}
Let $S\subseteq RS^{d-1}$ be a spherical segment of direction $\frac{b}{\lvert b \rvert}$, with $b \in \mathbb{Z}^d$ and height $0\le h\le 2R$. Then 
\[\psi \le \kappa_d(R)(1+\lvert b \rvert h).\]
\end{lemma}

\Cref{lem:lemma1}, although simple, provides the best 
 upper bound using slicing methods for a segment of rational direction; we will use it to prove the next lemma. \Cref{lem:lemma1} will be proved in the next section.
\begin{lemma} 
\label{lem:lemma2}
Let $S\subseteq RS^{d-1}$ be a spherical segment with direction $\beta$ and opening angle $\theta$. For any $a\in \mathbb{Z}^d$, the maximal number of lattice points lying on $S$ satisfies 
\[\psi \ll \kappa_d(R)(1+R\lvert a \rvert(\theta+\phi)) \;\;\; \text{as} \;\; \theta \to 0 \]
where $0\le \phi \le \pi$ is the angle between $\beta$ and $a$. The implied constant is absolute. 
\end{lemma}

We will prove \Cref{lem:lemma2} in the next section. We will try to find an $a \in \mathbb{Z}^d$ that optimises $\rvert a \lvert$ and $\phi$ simultaneously using Diophantine approximation. The main ingredient is the following.

\begin{theorem}(Dirichlet's theorem) \cite[Theorem 1A, p.27]{diophantine}.

 \label{dirichlet}
Let $\xi_i \in \mathbb{R}$ for $i=1,2 \ldots d$ , and $H\in \mathbb{N}$.
Then there exist $p_i \in \mathbb{Z}$ and $1\le q \le H^d$ such that 
\[\left\lvert \xi_i-\frac{p_i}{q}\right\rvert \le \frac{1}{qH} \;\;\;\;   \forall \; i=1,2 \ldots d.\]
\end{theorem}

Note that $\left \lvert \frac{a}{\lvert a \rvert}-\beta \right\rvert=2\sin{(\phi/2)} \sim \phi \; \text{as} \;\; \phi \to 0$. The next lemma will provide a good simultaneous bound for $\lvert a \rvert$ and $\phi$.

\begin{lemma}
 \label{lem:lemma3}
For all $\beta \in \mathbb{R}^d$, $\lvert \beta \rvert=1$, and integers $H\ge 1$, there exists some $a \in \mathbb{Z}^d$ such that 
\[\vert a \rvert \ll H^{d-1} \;\; \text{and} \;\; \left\lvert \frac{a}{\vert a\rvert}-\beta \right\rvert \ll \frac{1}{\lvert a \rvert H},\]
where $\ll$ is understood with respect to $H$, and the implied constants are absolute.
\end{lemma}
We will prove \Cref{lem:lemma3} in the next section.

\begin{proof}[Proof of \Cref{mainprop} assuming all the lemmas]
We follow the proof of the $3$-dimensional case by Maffucci \cite[Prop. 6.2]{maffucci}. Given a segment with direction $\beta$ we choose $a\in \mathbb{Z}^d$ satisfying  \Cref{lem:lemma3}, for some $H$ to be determined. Then one has
 \begin{equation*}
 \left\lvert \beta-\frac{a}{\lvert a\rvert} \right\rvert=2\sin(\phi/2)\sim \phi \;\;\;\; \text{as} \; \phi \to 0.
 \end{equation*}
Later we will choose $H$ so that $\theta \to 0$ implies $\phi \to 0$, in which case by \Cref{lem:lemma2} gives 
\begin{equation} \label{hola}
\psi \ll \kappa_d(R)\left(1+R\lvert a \rvert \theta + R\lvert a \rvert \left \lvert \beta-\frac{a}{\lvert a\rvert} \right \rvert \right)  \text{ as } \theta \to 0.
\end{equation}
It follows from \Cref{lem:lemma3} that
\begin{equation} \label{eq2.2}
\lvert a \rvert \theta + \lvert a \rvert \left\lvert \beta-\frac{a}{\lvert a\rvert} \right\rvert \ll \theta H^{d-1}+\frac{1}{H}.
\end{equation}
Setting $H=\left\lceil \theta^{-1/d}\right\rceil$ in \eqref{eq2.2}, so that $H=\mathcal{O}(\theta^{-1/d})$, we obtain
\begin{equation*}
\lvert a \rvert \theta + \lvert a \rvert \left \lvert \beta-\frac{a}{\lvert a\rvert} \right \rvert \ll \theta^{1/d}.
\end{equation*}
Therefore, since the implied constants in \eqref{hola} and \eqref{eq2.2} are absolute, if follows that
\begin{equation*}
\psi \ll \kappa_d(R)(1+R\theta^{1/d}) \text{ as } \theta \to 0.
\end{equation*}
To complete the proof we have that for our choice of $H$,
\begin{equation*}
2\sin(\phi/2)= \left\lvert \beta-\frac{a}{\lvert a\rvert} \right\rvert \ll \frac{1}{\lvert a \rvert H} \le \frac{1}{H} \sim \theta^{1/d} \to 0 \;\; \text{as} \; \theta \to 0.
\end{equation*}
Therefore $\phi \to 0$ as $\theta \to 0$, because $0\le \phi \le \pi$.
\end{proof}

\section{Proofs of the Lemmas}

\begin{proof}[Proof of \Cref{lem:lemma1}]
Since $b \in \mathbb{Z}^d$, all lattice points in the segment lie on a (rational) hyperplane of the form $b\cdot x=n\in \mathbb{Z}$. The distance between planes is 

\[
\left\lvert \frac{nb}{\lvert b \rvert^2} - \frac{(n+1)b}{\lvert b \rvert^2} \right\rvert=\frac{1}{\lvert b\rvert}
\]

 \noindent and there are at most 
\[1+\left\lfloor \frac{h}{1/ \lvert b \rvert}\right\rfloor \le 1+\lvert b \rvert h\]

\noindent hyperplanes intersecting the segment, each of them with at most $\kappa_d(R)$ lattice points in them; thus 
 the lemma follows.
\end{proof}

A slightly different proof of \Cref{lem:lemma1} when $d=3$ can be found in \cite[Proposition 6.3]{maffucci}.

\begin{proof}[Proof of \Cref{lem:lemma2}]
Let $S=T_1\setminus T_2$ with respective radii $r_1$ and $r_2$. We will construct $S'=T_1'\setminus T_2'$ having direction proportional to $a \in \mathbb{Z}^d$, such that $S\subseteq S'$, then $\psi$ can be bounded above using \Cref{lem:lemma1} on $S'$. 
For this it will suffice to construct $S'$ so that $T_2'\subseteq T_2$ and $T_1\subseteq T_1'$. Let $r'_1,r'_2$ be the radii of $T'_1$ and $T'_2$ respectively. We claim that taking 
\begin{equation*} 
r_2' = r_2-2R\sin(\phi/2) 
\end{equation*} and
\begin{equation*}
 r_1'^2 =r_1^2+4Rr_1\sin(\phi/2)+4R^2\sin(\phi/2)^2
\end{equation*}
 satisfies the conditions. A point inside $T_2'$ is of the form $R\frac{a}{\lvert a \rvert}+ v$ with $\lvert v \rvert \le r_2'$, hence
\begin{equation*}
\begin{aligned}
\left\lvert R\beta -\left(R\frac{a}{\lvert a \rvert}+v\right)\right\rvert^2 &=R^2\left\lvert \beta-\frac{a}{\lvert a \rvert}\right\rvert^2+\lvert v\rvert^2+2Rv\cdot\left(\beta-\frac{a}{\lvert a \rvert}\right) \\
&\le 4R^2\sin(\phi/2)^2+\lvert v \rvert^2+2R\lvert v\rvert \left\lvert \beta-\frac{a}{\lvert a \rvert}\right\rvert \\
&\le 4R^2\sin(\phi/2)^2+r_2'^2+4R\sin(\phi/2)r_2'=r_2^2,
\end{aligned}
\end{equation*}
where we have used that $\left\lvert \beta -\frac{a}{\lvert a \rvert}\right\rvert=2\sin(\phi/2)$. This shows that $T_2'\subseteq T_2$. Now let $R\beta+v$ with $\lvert v \lvert\le r_1$ be a point of $T_1$. We have
\begin{equation*}
\begin{aligned}
\left\lvert \left(R\beta+v\right) -R\frac{a}{\lvert a \rvert} \right\rvert^2 &\le 4R^2\sin(\phi/2)^2+\lvert v \rvert^2+2R\lvert v\rvert \left\lvert \beta-\frac{a}{\lvert a \rvert}\right\rvert \\
&\le 4R^2\sin(\phi/2)^2+r_1^2+4Rr_1\sin(\phi/2)=r_1'^2,
\end{aligned}
\end{equation*}
so that $T_1\subseteq T_1'$ as desired. Now, according to \eqref{height}, the height of $S'$ is 
\begin{equation*}
h'=\frac{1}{2R}(r_1'^2-r_2'^2)=\frac{1}{2R}\left(r_1^2-r_2^2+4R\sin(\phi/2)(r_1+r_2)\right).
\end{equation*}
Recall that $r_i=2R\sin(\theta_i/4)$ \eqref{rtheta}, and $\theta=\theta_1-\theta_2$. Therefore,
\begin{equation} \label{h'}
h'=2R(\sin(\theta_1/4)^2-\sin(\theta_2/4)^2)+4R\sin(\phi/2)(\sin(\theta_1/4)+\sin(\theta_2/4)).
\end{equation}
\noindent Now, $\theta_1/4=\theta_2/4+\theta/4$, so that using the Taylor expansion of $\sin(x)$ yields
\begin{equation*} 
\sin(\theta_1/4)=\sin(\theta_2/4)\cos(\theta/4)+\cos(\theta_2/4)\sin(\theta/4)\frac{\theta}{4}+\mathcal{O}(\theta^2),
\end{equation*}
implying
\begin{equation} \label{theta}
\sin(\theta_1/4)^2=\sin(\theta_2/4)^2+\sin(\theta_2/2)\frac{\theta}{4}+\mathcal{O}(\theta^2).
\end{equation}

\noindent It follows from \eqref{h'} and \eqref{theta} that
\begin{equation} \label{expression}
h'=2R\left(\sin(\theta_2/2)\frac{\theta}{4}+2\sin(\phi/2)\left(\sin(\theta_1/4)+\sin(\theta_2/4)\right)+\mathcal{O}(\theta^2)\right).
\end{equation}
We have that $\sin(\phi/2)\le \phi/2$, and all the other sines in \eqref{expression} are bounded above by $1$, thus 
\begin{equation*}
h'\ll R(\theta+\phi)
\end{equation*}
as $R$ tends to infinity and $\theta$ tends to $0$, and the implied constant is absolute. By \Cref{lem:lemma1}, the number of lattice points in $S'$ is no greater than
\begin{equation*}
\kappa_d(R)(1+\lvert a\rvert h')\ll \kappa_d(R)(1+R\lvert a \rvert(\theta+\phi)),
\end{equation*}
 which is an upper bound for $\psi$ since $S\subseteq S'$.
\end{proof}
We now prove an auxiliary lemma, which will be used in the proof of \Cref{lem:lemma3}.
\begin{lemma}
\label{lem:lemma4}
Let $\alpha$, $\beta$ be two non-zero vectors of $\mathbb{R}^n$. Then
\[\left\lvert \frac{\alpha}{\lvert \alpha \lvert}-\frac{\beta}{\lvert \beta \rvert} \right\rvert \le 2\frac{\lvert \alpha-\beta \rvert}{\lvert \alpha \rvert}. \]

\end{lemma}
\begin{proof}
This is an easy application of the triangle inequality. We have
\begin{equation*} 
\left\lvert \frac{\alpha}{\lvert \alpha \lvert}-\frac{\beta}{\lvert \beta \rvert} \right\rvert = \frac{1}{\lvert \alpha \rvert} \left \lvert \alpha -\frac{\lvert \alpha \rvert}{\lvert \beta \rvert}\beta \right\rvert
\end{equation*}
and
\begin{equation*}
\left \lvert \alpha -\frac{\lvert \alpha \rvert}{\lvert \beta \rvert}\beta \right\rvert \le \lvert \alpha-\beta \rvert + \left\lvert \beta -\frac{\lvert \alpha \rvert}{\lvert \beta \rvert}\beta \right\rvert=\lvert \alpha-\beta\rvert + \left\lvert \lvert \beta \lvert -\lvert \alpha \rvert \right \rvert \le 2\lvert \alpha-\beta\rvert.
\end{equation*}

\end{proof}
Next, we prove \Cref{lem:lemma3}.

\begin{proof}[Proof of \Cref{lem:lemma3}]
Without loss of generality we assume that $\lvert \beta_1 \rvert=\max( \lvert \beta_i \rvert \;  i=1,2\ldots d)$. Define $\xi_i=\frac{\beta_i}{\beta_1}$ for $i=2,3\ldots d$, so that $\lvert \xi_i \rvert \le 1$ for all $i$. 
By Dirichlet's theorem (\Cref{dirichlet}) there exist $1\le q\le H^{d-1}$ and $p_i \in \mathbb{Z}$ for $i=2,3,\ldots,d$ such that 
\begin{equation} \label{di}
\left\lvert \xi_i -\frac{p_i}{q} \right\rvert \le \frac{1}{qH} \;\; \text{for} \; i=2,3,\ldots,d.
\end{equation}
 Let $a=(q,p_2,\ldots,p_d)$. We have $\left\lvert \frac{p_i}{q}\right\rvert \le 1+\frac{1}{qH}\le 2$, so that $\lvert p_i \rvert \le 2q$. Hence
\begin{equation} \label{bound1}
\vert a \rvert^2=q^2+p_2^2+\ldots+p_d^2\le (4d-3)q^2 \; \implies \vert a \rvert \le (4d-3)^{1/2}q\ll H^{d-1}.
\end{equation}

\noindent Now let $d=\frac{\beta_1}{q}a$, thus $\frac{d}{\lvert d \rvert}=\frac{a}{\lvert a \rvert}$, so that by \Cref{lem:lemma4} and \eqref{di}
\begin{equation*}
\left\lvert \beta-\frac{a}{\lvert a\rvert} \right\rvert=\left\lvert \beta-\frac{d}{\lvert d\rvert}\right\rvert \le 2\lvert \beta -d \lvert =2\lvert \beta_1\rvert \left(\sum_{i=2}^{d}{\left(\xi_i-\frac{p_i}{q}\right)^2}\right)^{1/2} \le \frac{2\sqrt{d-1}}{qH}.
\end{equation*}
 Since $\lvert a \rvert \le (4d-3)^{1/2}q$, it follows that
 \begin{equation} \label{3.5}
 \left\lvert \beta-\frac{a}{\lvert a\rvert} \right\rvert \le \frac{2(4d^2-7d+3)^{1/2}}{\lvert a \rvert H}\ll \frac{1}{\lvert a\rvert H}.
 \end{equation}
 Both implied constants in \eqref{bound1} and \eqref{3.5} are absolute.

\end{proof}

\section{Rational quotients: proof of Theorem 1.5}

We now prove a generalisation of \Cref{lem:lemma3} for the case when the direction of the segment has rational quotients (recall \Cref{def3}). We focus on the case when the number of rational quotients $s$ is less than or equal to $d-2$. For $s=d-1$ there exists a non-zero $k$ such that $k\beta \in \mathbb{Z}^d$. Then by \Cref{lem:lemma1}
\[
\psi \ll_\beta \kappa_d(R)(1+h).
\]
Arguing as in the proof of \Cref{lem:lemma2} we find that
\[
h=2R(\sin(\theta_1/4)^2-\sin(\theta_2/4)^2)=2R(\sin(\theta_2/4)\frac{\theta}{4} + \mathcal{O}(\theta^2))\sim R\theta \text{ as } \theta \to 0
\]
Thus, this gives the same bound as \Cref{thm1.2}, and no Diophantine approximation is needed. 
\begin{lemma} \label{cor1} 
Let $\beta \in \mathbb{R}^d$, $\lvert \beta \rvert=1$, have $s$ rational quotients, $1\le s \le d-2$, and let $H\ge 1$ be an integer. Then there exists $a \in \mathbb{Z}^d$ such that
\[ \lvert a \rvert \ll_\beta H^{d-1-s} \;\; \text{and} \;\; \left\lvert \frac{a}{\lvert a\rvert}-\beta \right\rvert \ll_\beta \frac{1}{\lvert a\rvert H},\]
where $\ll_\beta$ is understood with respect to $H$.

\begin{proof}
Let $k\in \mathbb{R}\setminus \{0\}$ such that $k\beta$ has $s+1$ rational coordinates. Define $\xi_i=k\beta_i$ for $i=1,2,\ldots d$, so that we have $\lvert \xi_i \rvert \le \lvert k \rvert$ for all $i$. 
 Without loss of generality assume that $\xi_i$ are rational for $i=1,2,\ldots,s+1$. Then
 \begin{equation*}
\xi_i=\frac{m_i}{n_i } \in \mathbb{Q} \;\;\; \text{for} \;\;\; 1\le i\le s+1.
 \end{equation*}
 By Dirichlet's theorem (\Cref{dirichlet}) there exist $1\le q'\le H^{d-1-s}$ and $p_i \in \mathbb{Z}$ such that
 \begin{equation} \label{4.1}
 \left\lvert \xi_i -\frac{p_i}{q'} \right\rvert \le \frac{1}{q'H} \;\;\; \text{for} \;\;\; s+2 \le i \le d.
 \end{equation}
Let $m=\prod_{i=1}^{s+1}{n_i}$, $q=q'm$, and $a=\left(\frac{qm_1}{n_1},\frac{qm_2}{n_2}\ldots ,\frac{qm_{s+1}}{n_{s+1}},mp_{s+2},\ldots,mp_d\right) \in \mathbb{Z}^d$. Then
 \begin{equation} \label{4.2}
 \lvert a \rvert=q \left( \sum_{i=1}^{s+1} \xi^2_i +\sum_{i=s+2}^{d} \frac{p^2_i}{q'^2} \right)^{1/2} \le (1+\lvert k \rvert )dq\le (1+\lvert k \rvert)dmH^{d-1-s}.
 \end{equation}
 \noindent Now let $d=\frac{1}{kq}a$, then $\frac{d}{\lvert d \rvert}=\frac{a}{\lvert a \rvert}$, so that by \Cref{lem:lemma4} and \eqref{4.1}
 \begin{equation*}
 \begin{aligned}
 &\left\lvert \beta-\frac{a}{\lvert a\rvert} \right\rvert=\left\lvert \beta-\frac{d}{\lvert d\rvert}\right\rvert \le 2\lvert \beta -d \lvert \\
 &=\frac{2}{\lvert k\rvert} \left(\sum_{i=2}^{s+1}{\left(\xi_i-\frac{m_i}{n_i}\right)^2}+\sum_{i=s+2}^{d}{\left(\xi_i-\frac{p_i}{q'}\right)^2}\right)^{1/2} \le \frac{2}{\lvert k\rvert} \frac{\sqrt{d-1-s}}{q'}\frac{1}{H}.
 \end{aligned}
 \end{equation*}
 Therefore, using \eqref{4.2}, we obtain
 \begin{equation*}
 \left\lvert \beta-\frac{a}{\lvert a\rvert} \right\rvert \le \frac{2(1+\lvert k\rvert)dm\sqrt{(d-1-s)}}{\lvert k \rvert \lvert a \rvert H}\ll_\beta \frac{1}{\lvert a\rvert H}.
 \end{equation*}
\end{proof}
\end{lemma}

This enables us to prove \Cref{thm1.2}, in essentially the same way as \Cref{mainprop}.

\begin{proof}[Proof of \Cref{thm1.2}]
In view of the paragraph at the beginning of the section we can assume that $1\le s \le d-2$. Given a segment with direction $\beta$ having $s$ rational quotients, we choose $a\in \mathbb{Z}^d$ satisfying  \Cref{cor1}, for some $H$ to be determined. Then
 \begin{equation*}
 \left\lvert \beta-\frac{a}{\lvert a\rvert} \right\rvert=2\sin(\phi/2)\sim \phi \;\;\;\; \text{as} \; \phi \to 0.
 \end{equation*}
Later we will choose $H$ so that $\theta \to 0$ implies $\phi \to 0$, in which case \Cref{lem:lemma2} gives
\begin{equation*}
\psi \ll \kappa_d(R)\left(1+R\lvert a \rvert \theta + R\lvert a \rvert \left \lvert \beta-\frac{a}{\lvert a\rvert} \right \rvert \right) \text{ as } \theta \to 0.
\end{equation*}
From \Cref{cor1}, it follows that
\begin{equation} \label{hey}
\lvert a \rvert \theta + \lvert a \rvert \left\lvert \beta-\frac{a}{\lvert a\rvert} \right\rvert \ll_\beta \theta H^{d-1-s}+\frac{1}{H}.
\end{equation}
Setting $H=\left \lceil \theta^{\frac{-1}{d-s}} \right \rceil$ in \eqref{hey}, so that $H=\mathcal{O}\left(\theta^{\frac{-1}{d-s}}\right)$, we obtain
\begin{equation*}
\lvert a \rvert \theta + \lvert a \rvert \left\lvert \beta-\frac{a}{\lvert a\rvert} \right \rvert \ll_\beta \theta^{\frac{1}{d-s}}.
\end{equation*}
Therefore,
\begin{equation*}
\psi \ll_\beta \kappa_d(R)\left(1+R\theta^{\frac{1}{d-s}}\right) \text{ as } \theta \to 0.
\end{equation*}
Finally, for our choice of $H$ we obtain
\begin{equation*}
2\sin(\phi/2)= \left\lvert \beta-\frac{a}{\lvert a\rvert} \right\rvert \ll_\beta \frac{1}{\lvert a \rvert H} \le \frac{1}{H} \sim \theta^{\frac{1}{d-s}} \to 0 \;\; \text{as} \; \theta \to 0.
\end{equation*}
Thus $\phi \to 0$ as $\theta \to 0$ since $0\le \phi \le \pi$.
\end{proof}

\renewcommand{\abstractname}{Acknowledgements}
\begin{abstract}
Thanks to Riccardo Maffucci for being my supervisor in this project, who has been indispensable in every sense. Thanks also to the LMS and Oxford University Math Institute for providing funding for this project. I would also like to thank Ze\'{e}v Rudnick and R. Heath-Brown for generously answering my questions about their work.
\end{abstract}

\medskip

\bibliographystyle{plain}
\bibliography{refs}

\begin{thebibliography}{10}

\bibitem{benatar-maffucci}
Jacques Benatar and Riccardo~W Maffucci.
\newblock {Random Waves On $\mathbb{T}^3$: Nodal Area Variance and Lattice
  Point Correlations}.
\newblock {\em International Mathematics Research Notices},
  2019(10):3032--3075, 09 2017.

\bibitem{rudnick}
Jean Bourgain and Ze\'{e}v Rudnick.
\newblock Restriction of toral eigenfunctions to hypersurfaces and nodal sets.
\newblock {\em Geometric and Functional Analysis}, 22:878--937, 2012.

\bibitem{BSR}
Jean {Bourgain}, Peter {Sarnak}, and Ze{\'e}v {Rudnick}.
\newblock {Local statistics of lattice points on the sphere}.
\newblock {\em arXiv e-prints}, page arXiv:1204.0134, Mar 2012.

\bibitem{bourgain_gauss}
Jean {Bourgain} and Nigel {Watt}.
\newblock {Mean square of zeta function, circle problem and divisor problem
  revisited}.
\newblock {\em arXiv e-prints}, page arXiv:1709.04340, Sep 2017.

\bibitem{cilleruelocordoba}
Javier Cilleruello and Antonio Cordoba.
\newblock Trigonometric polynomials and lattice points.
\newblock {\em Proceedings of the American Mathematical Society}, 114(4), 1992.

\bibitem{cilleruelogranville}
Javier Cilleruelo and Andrew Granville.
\newblock Close lattice points on circles.
\newblock {\em Canadian Journal of Mathematics-journal Canadien De
  Mathematiques - CAN J MATH}, 61, 12 2009.

\bibitem{Duke1988}
W.~Duke.
\newblock Hyperbolic distribution problems and half-integral weight {M}aass
  forms.
\newblock {\em Inventiones mathematicae}, 92(1):73--90, Feb 1988.

\bibitem{grosswald}
Emil Grosswald.
\newblock {\em Representations of Integers as Sums of Squares}.
\newblock Springer-Verlag, 1985.

\bibitem{hardy-wright}
G.H. Hardy and E.M. Wright.
\newblock {\em An introduction to the theory of numbers}.
\newblock Oxford University Press, 2008.

\bibitem{heathbrownsphere}
R.~Heath-Brown.
\newblock Lattice points in the sphere.
\newblock {\em Number theory in progress}, (2), 1999.

\bibitem{jarnik}
V.~Jarnik.
\newblock Uber die gitterpunkte auf konvexen kurven.
\newblock {\em Math. Z}, 24, 1926.

\bibitem{wigman}
Manjunath {Krishnapur}, P\"ar {Kurlberg}, and Igor {Wigman}.
\newblock {Nodal length fluctuations for arithmetic random waves.}
\newblock {\em {Ann. Math. (2)}}, 177(2):699--737, 2013.

\bibitem{maffucci}
Riccardo~W. Maffucci.
\newblock Nodal intersections for random waves against a segment on the
  $3$-dimensional torus.
\newblock {\em Journal of Functional Analysis}, 272(12):5218--5254, 2017.

\bibitem{leray}
Ferenc Oravecz, Ze\'ev Rudnick, and Igor Wigman.
\newblock The {L}eray measure of nodal sets for random eigenfunctions on the
  torus.
\newblock {\em Annales de l'Institut Fourier}, 58(1):299--335, 2008.

\bibitem{rossi-wigman}
Maurizia {Rossi} and Igor {Wigman}.
\newblock {Asymptotic distribution of nodal intersections for arithmetic random
  waves}.
\newblock {\em Nonlinearity}, 31(10):4472, Oct 2018.

\bibitem{diophantine}
Wolfgang~M. Schimdt.
\newblock {\em Diophantine approximation}.
\newblock Springer-Verlag, 2 edition, 1993.

\bibitem{szego}
G.~Szego.
\newblock Beitrage zur theorie de laguerreschen polynome.
\newblock {\em Zahlentheoretische Anwendungen}, 1926.

\bibitem{walfisz}
A.~Walfisz.
\newblock On the class-number of binary quadratic forms.
\newblock {\em Math. Tbilissi}, 11, 1942.

\end{thebibliography}

\end{document}